\documentclass[a4paper,onecolumn]{ieeeconf}
\usepackage{generic}
\usepackage{cite}
\usepackage{amsmath,amssymb,amsfonts}
\usepackage{mathrsfs}
\usepackage{algorithmic}
\usepackage{graphicx}
\usepackage{tikz}
\usepackage{pgfplots}
\pgfplotsset{compat=1.18}
\usepackage{empheq}
\usepackage{comment}
\newtheorem{theorem}{Theorem}
\newtheorem{proposition}{Proposition}
\newtheorem{corollary}{Corollary}
\newtheorem{property}{Property}
\newtheorem{remark}{Remark}
\newtheorem{example}{Example}
\newtheorem{definition}{Definition}
\newtheorem{ass}{Assumption}
\newtheorem{lem}{Lemma}
\newenvironment{assumption}{\begin{ass}}{\hfill $\bullet$ \end{ass}}

\newenvironment{lemma}{\begin{lem}}{\hfill $\square$ \end{lem}}
\newcommand{\sign}{\operatorname{sign}}
\usepackage{algorithm,algorithmic}
\usepackage{hyperref}
\hypersetup{hidelinks=true}
\usepackage{textcomp}
\begin{document}
\title{\bf Global $L^1$ Exponential Observer for the Quasilinear Diffusion Equation}
\author{M. C. Belhadjoudja and K. A. Morris
\thanks{M. C. Belhadjoudja is with Université Grenoble Alpes, CNRS, Grenoble-INP, GIPSA-lab, F-38000, Grenoble, France (email: mohamed.belhadjoudja@gipsa-lab.fr)}
\thanks{K. A. Morris is with the Department of Applied Mathematics, University of Waterloo, 200 University Avenue West, Waterloo, ON, Canada, N2L 3G1 (email: kmorris@uwaterloo.ca)}}

\maketitle
\pagenumbering{gobble}

\begin{abstract}
A nonlinear observer is designed for a one-dimensional quasilinear diffusion equation with Neumann boundary conditions, using as an output the state at one end of the spatial domain. The observer consists of a copy of the system, with a nonlinear (in the output) correction term at the boundary. A Lyapunov-based analysis is performed to ensure global exponential stability of the zero equilibrium for the estimation error dynamics in the $L^1$ norm. The result holds for a general nonlinear diffusivity, under the assumption that it is lower-bounded by a positive constant, i.e., that the diffusion equation is uniformly parabolic.

\end{abstract}

\section{Introduction}

The quasilinear diffusion equation arises in a wide range of applications in which the rate of diffusion depends on the state. Examples include heat transfer in high-temperature plasmas \cite[Pages 652--684]{plasma}, porous-media flows \cite{porous1}, and population dynamics \cite[Section 2.4]{vazquez}.

In these applications, the distributed state cannot be measured throughout the spatial domain. State reconstruction, for monitoring and control, must therefore rely on a limited number of measurements, typically available only at accessible locations such as the boundary. This motivates the development of observers capable of estimating the state from boundary measurements. In this paper, we consider this problem for a one-dimensional quasilinear diffusion equation with Neumann boundary conditions, when the available measurement is the state at one end of the spatial domain. 

When designing an observer for a nonlinear dynamical system, one may seek either local or global asymptotic or exponential stability guarantees for the zero equilibrium of the estimation error dynamics. In this paper, we seek global guarantees for general nonlinear diffusivity, a problem that, to the best of our knowledge, has remained open in the literature. For constant diffusivity, there is a rich body of results, including several well-established observer design methods such as backstepping, Kalman filtering, and pole placement, to name a few. In contrast, for quasilinear diffusion, in which the nonlinearity appears in the highest-order derivative term, the only available result known to us that does not linearize the equation is our previous work \cite{camil_back}, where we estimate a bounded region of attraction for a backstepping observer for an equation with general nonlinear diffusivity. Although stabilization results are available for some classes of quasilinear parabolic equations, such as \cite{camil_finite_2,hugo}, extending them to the global observer design problem does not seem straightforward to us.

Our observer consists of a copy of the system with a boundary correction term, that is constructed using the difference of the nonlinear diffusivity potential evaluated at the measured and estimated boundary values. We establish global exponential stability of the zero equilibrium for the estimation error dynamics in the $L^1$ norm. The result requires only that the nonlinear diffusivity be bounded from below by a strictly positive constant. That is, the system must be uniformly parabolic. 

The main difficulty in the quasilinear setting is that a direct Lyapunov analysis based on the squared $L^2$ norm of the estimation error is not well suited to a general nonlinear diffusivity. Indeed, differentiating this functional along the error dynamics produces terms involving spatial derivatives of the nonlinear diffusivity. Controlling these terms generally requires additional structural assumptions on the diffusivity, such as suitable upperbounds or growth conditions, which would restrict the generality of the result. We overcome this difficulty by exploiting the particular structure of the nonlinear boundary correction term and by constructing a Lyapunov functional from the spatial primitive of the estimation error rather than from the error itself. This choice leads to a dissipation identity in which the nonlinear diffusivity enters only through a positive averaged derivative. The lowerbound on the diffusivity therefore suffices to establish exponential decay of the Lyapunov functional. However, this Lyapunov functional is not equivalent to the norm in which we seek to establish global stability. We address this by deriving a separate contraction property for the $L^1$ norm of the estimation error. This is achieved by regularizing the absolute value function and exploiting the fact that the nonlinear diffusivity is strictly increasing, which ensures that the estimation error and the difference of the diffusivity potentials have the same sign. Combining this contraction property with the decay of the Lyapunov functional allows us to prove global exponential stability in $L^1$.

\section{Notations}
We denote by $\mathbb{R}_{>0}$ (resp., by $\mathbb{R}_{\geq 0}$) the set of positive (resp., of nonnegative) real numbers. Given $p\in [1,+\infty)$ and a function $z:[0,1]\to \mathbb{R}$, we let
$|z|_{L^p}^p := \int_0^1 |z(x)|^p dx$. Similarly, given a function
$z:[0,1]\times[0,+\infty)\to\mathbb{R}$ and $t\geq 0$, we let
$|z(\cdot,t)|_{L^p}^p := \int_0^1 |z(x,t)|^p dx$. Furthermore, we let $|z(\cdot,t)|_{\infty}:=\max_{x \in [0,1]}|z(x,t)|$. Given $\beta \in (0,1)$, we denote by $\mathscr{H}^{2+\beta}([0,1])$ 
the space of twice continuously-differentiable functions $z: [0,1] \rightarrow \mathbb{R}$ such that $z''$, the second-order derivative of $z$, is H\"{o}lder continuous with exponent $\beta$, i.e., there exists $C>0$ such that 
$ |z''(x_1)-z''(x_2)| \leq C |x_1-x_2|^{\beta}$ for all $x_1,x_2 \in [0,1]$.  
We denote by $\mathscr{H}_{loc}^{2+\beta,1+\beta/2}([0,1]\times [0,+\infty))$ the space of functions $z :[0,1] \times [0,+\infty) \to \mathbb{R}$ such that the restriction of $z$ to any bounded subset $[0,1] \times L \subset [0,1]\times [0,+\infty)$ belongs to $\mathscr{H}^{2+\beta,1+\beta/2}([0,1] \times L)$, the space of functions $f: [0,1]\times L\to \mathbb{R}$ such that $f_{xx}$ and $f_t$
are H\"older continuous in $x$ with exponent $\beta$ uniformly in $t$
and H\"older continuous in $t$ with exponent $\beta/2$ uniformly
in $x$. Given a function $z:[0,1]\times [0,+\infty)\to \mathbb{R}$, we let $\|z\|_{\infty} := \sup_{(x,t)\in [0,1]\times [0,+\infty)} |z(x,t)|$. Finally, given $a\in \mathbb{R}$, we let
\begin{equation*}
\sign(a) := 
\left\lbrace 
\begin{aligned}
&1 &&\text{if $a>0$},\\
&0 &&\text{if $a=0$}, \\
&-1 &&\text{if $a<0$}.
\end{aligned}
\right.
\end{equation*}

\section{System Description}

Consider the quasilinear diffusion equation
\begin{empheq}[left=\Sigma: \left\{,right=\right.]{align}
z_t &= \Psi(z)_{xx}, \label{sigma_1_z}\\
\Psi(z)_x \big|_{x=0} &= \Psi(z)_x \big|_{x=1} = 0, \label{sigma_21}\\
z(x,0) &= z_o(x), \label{sigma_3_z}
\end{empheq}
where $z(x,t)\in \mathbb{R}$ is the state variable with $x\in [0,1]$ the space variable and $t\geq 0$ the time variable, $z_o$ is the initial condition, and $\Psi : \mathbb{R}\to \mathbb{R}$ is referred to as the \textit{diffusivity potential}, whereas its derivative $\Psi'$ will be referred to as the \textit{diffusivity}. 

\begin{definition}\label{def_classical}
Let $T\in (0,+\infty) \cup \{+\infty\}$. A function $z: [0,1]\times [0,T)\to \mathbb{R}$ is said to be a classical solution to $\Sigma$ on $[0,T)$ if: 
\begin{itemize}
    \item $z\in \mathscr{H}_{loc}^{2+\beta, 1+\beta/2}([0,1]\times [0,T))$ for some $\beta \in (0,1)$.
    \item \eqref{sigma_1_z} holds for all $(x,t) \in (0,1)\times (0,T)$. 
    \item \eqref{sigma_21} holds for all $t\in [0,T)$. 
    \item \eqref{sigma_3_z} holds for all $x\in [0,1]$.
\end{itemize}
A forward-complete classical solution to $\Sigma$ is any classical solution to $\Sigma$ on $[0,+\infty)$.
\hfill $\bullet$
\end{definition}

The following assumptions guarantee the existence and uniqueness of a forward-complete classical solution.

\begin{assumption}\label{ass_1}
$\Psi \in \mathscr{C}^3(\mathbb{R})$, and there exists a constant $\mu$ satisfying
    $$\Psi' \geq \mu > 0. $$
\end{assumption}

\begin{assumption}\label{ass_o}
    There exists $\beta \in (0,1)$ such that
    $z_o\in \mathscr{H}^{2+\beta}([0,1])$. Furthermore, $$\Psi(z_o)_x\big|_{x=0}=\Psi(z_o)_x\big|_{x=1}=0.$$
\end{assumption}

Under Assumptions \ref{ass_o}-\ref{ass_1}, one has the following result. 

\begin{proposition}[Page 491, Theorem 7.4 in \cite{ladyzen}]
    Under Assumptions \ref{ass_1} and \ref{ass_o}, there exists a unique forward-complete classical solution to $\Sigma$.
    \hfill $\square$
\end{proposition}

Let $z$ be the unique forward-complete classical solution to $\Sigma$. Given the regularity of $z$, we can consider the output
\begin{align}
y(t) := z(1,t) \qquad \forall t\geq 0. 
\end{align}
The objective is to design, using the output $y$, an observer for $\Sigma$ that exponentially estimates $z$ in the sense of the $L^1$ norm, regardless of the size of the estimation error at $t= 0$.

\section{Observer Design}

We propose the observer 
\begin{equation*}
\hat{\Sigma} : \left\lbrace 
\begin{aligned}
\hat{z}_t &= \Psi(\hat{z})_{xx}, \\
\Psi(\hat{z})_x \big|_{x=0} &=0, \\
\Psi(\hat{z})_x \big|_{x=1} &= k\left(\Psi(y(t))-\Psi(\hat{y}(t))\right), \\
\hat{z}(x,0) &= \hat{z}_o(x),
\end{aligned}
\right.
\end{equation*}
where $\hat{z}(x,t) \in \mathbb{R}$ is the observer's state, $\hat{y}(t):=\hat{z}(1,t)$ for all $t\geq 0$, $k>0$ is a freely-tunable observer gain, and $\hat{z}_o$ is the observer's initial condition satisfying the following regularity and compatibility conditions, necessary for the existence of a classical solution. 
\begin{assumption}\label{ass_3}
$\hat{z}_o\in \mathscr{H}^{2+\beta}([0,1])$ where $\beta$ is the same as in Assumption \ref{ass_o}, and
\begin{align*}
\Psi(\hat{z}_o)_x\big|_{x=0} = 0, \ \ \Psi(\hat{z}_o)_x\big|_{x=1} = k\left(\Psi(y(0))-\Psi(\hat{y}(0))\right). 
\end{align*}
\end{assumption}

A direct application of \cite[Page 491, Theorem 7.4]{ladyzen} yields the following result. 

\begin{proposition}
Under Assumptions \ref{ass_1}-\ref{ass_3}, there exists a unique forward-complete classical solution to $\hat{\Sigma}$. 
\hfill $\square$
\end{proposition}

Define the estimation error 
\begin{align*}
e := z-\hat{z}.
\end{align*}
Subtracting $\hat{\Sigma}$ from $\Sigma$, one readily verifies that $e$ is the unique forward-complete classical solution to 
\begin{equation*}
\Sigma_e : \left\lbrace 
\begin{aligned}
e_t &= w_{xx}, \\
w &:= \Psi(z)-\Psi(z-e), \\
w_x(0,t) &=0, \\
w_x(1,t) &= -kw(1,t), \\
e(x,0) &= e_o(x) := z_o(x) -\hat{z}_o(x). 
\end{aligned}
\right.
\end{equation*}

\section{$L^1$ Global Exponential Stability}

In this section, we establish $L^1$ global exponential stability of the zero equilibrium $\{e=0\}$ for $\Sigma_e$. That is, we prove the following result. 

\begin{theorem}\label{thm1}
Consider the systems $\Sigma$ and $\hat{\Sigma}$ with $k>0$ and observer gain, and let Assumption \ref{ass_1} hold. Then, for any initial conditions $z_o$ and $\hat{z}_o$ satisfying the regularity and compatibility conditions in Assumptions \ref{ass_o} and \ref{ass_3}, the estimation error $e=z-\hat{z}$, that is the unique forward-complete classical solution to $\Sigma_e$, satisfies, for all $t\geq 0$, 
\begin{align}
\int_0^1 |e(x,t)| dx \leq \min \left\{1,C_1 \exp\left(-\lambda \mu t\right)\right\} \int_0^1|e_o(x)|dx, \label{stab_ineq}
\end{align}
where 
\begin{align*}
\lambda &:= \xi^2, \\
C_1 &:= \sqrt{\exp(1)\left(\xi^2+\xi\cot(\xi)\right)},
\end{align*}
with $\xi$ the unique solution in $(0,\pi/2)$ to 
$$\xi \tan(\xi) = k.$$
\hfill $\square$
\end{theorem}

\begin{remark}\label{rem1}
Since $\xi \in (0,\pi/2)$, then $\lambda \in (0,\pi^2/4)$. To determine the range of $C_1$, we analyze the variations of the function $f(\xi) := \xi \cot(\xi) + \xi^2$ on $(0,\pi/2)$. Using $\frac{d}{d\xi}\cot(\xi) = -\left(1+\cot^2(\xi)\right)$, we have
$f'(\xi) = \cot(\xi)+\xi-\xi\cot^2(\xi)$, so that, multiplying by $\sin^2(\xi)>0$,
\begin{align*}
\sin^2(\xi)f'(\xi) &= \sin(\xi)\cos(\xi) + \xi \left(\sin^2(\xi)-\cos^2(\xi)\right) \\
&= \frac{1}{2}\sin(2\xi)-\xi\cos(2\xi)=: g(\xi).
\end{align*}
Since $g(0)=0$ and $g'(\xi) = 2\xi\sin(2\xi)>0$ for all $\xi \in (0,\pi/2)$, we have $g>0$,
hence $f'>0$ on $(0,\pi/2)$. As $\lim_{\xi\to 0^+}\xi\cot(\xi)=1$ and $\cot(\pi/2)=0$, it
follows that $1<f(\xi)<\pi^2/4$ on $(0,\pi/2)$. Consequently,
\begin{align*}
C_1 \in \left(\exp\left(\frac{1}{2}\right),\exp\left(\frac{1}{2}\right)\frac{\pi}{2}\right).
\end{align*}
\hfill $\bullet$
\end{remark}

\subsection{Preliminary Lemmas}

In order to prove Theorem \ref{thm1}, we will need the lemmas established in this subsection. 

We start by proving the following technical lemma. 

\begin{lemma}\label{lem1}
Define the function 
\begin{align*}
\theta(x,t) := \int_0^1\Psi'\left(\hat{z}(x,t)+\tau e(x,t)\right)d\tau \geq \mu. 
\end{align*}
Then, 
\begin{align*}
w = \theta e.
\end{align*}
In particular, 
\begin{align*}
\sign(w) = \sign(e).
\end{align*}
\end{lemma}
\begin{proof}
The inequality $\theta \geq \mu$ is a direct consequence of Assumption \ref{ass_1}. To establish $w=\theta e$, note that 
\begin{align*}
\theta(x,t)e(x,t) &= e(x,t) \int_0^1 \Psi'\left(\hat{z}(x,t)+\tau e(x,t)\right)d\tau \\
&= \left[\Psi\left(\hat{z}(x,t)+\tau e(x,t)\right)\right]_{\tau=0}^{\tau=1} \\
&= \Psi(z(x,t)) - \Psi(\hat{z}(x,t)) \\
&= w(x,t).
\end{align*}
Finally, $\sign(w) = \sign(e)$ follows from $w=\theta e$ and $\theta \geq \mu >0$. 
\end{proof}

Next, we introduce the function 
\begin{align}
V(t) := \frac{1}{2}\int_0^1E(x,t)^2dx + \frac{1}{2k} M(t)^2,
\end{align}
where 
\begin{align}
E(x,t) := \int_0^x e(y,t)dy, \ M(t) := E(1,t). 
\end{align}
We derive a differential inequality on $V$ in the next lemma. 
\begin{lemma}\label{lem2}
It holds that 
\begin{align}
\dot{V}(t) = -\int_0^1\theta (x,t) e(x,t)^2 dx \leq - \mu \int_0^1 e(x,t)^2 dx. \label{V_dot}
\end{align}
\end{lemma}
\begin{proof}
Differentiating $V$ in time, note that we have 
\begin{align*}
\dot{V}(t) =&~ \int_0^1E(x,t) E_t(x,t) dx + \frac{1}{k}E(1,t)E_t(1,t) \\
=&~ \int_0^1 E(x,t)\left(\int_0^xe_t(y,t)dy \right)dx + \frac{1}{k}E(1,t)\int_0^1e_t(y,t)dy \\
=&~\int_0^1 E(x,t)\left(\int_0^xw_{yy}(y,t)dy \right)dx + \frac{1}{k}E(1,t)\int_0^1w_{yy}(y,t)dy.
\end{align*}
Applying integration by parts and using the boundary conditions $w_x(0,t)=E(0,t)=0$ and $w_x(1,t) = - k w(1,t)$, we obtain 
\begin{align*}
\dot{V}(t) =&~ \int_0^1 E(x,t)w_x(x,t)dx + \frac{1}{k}E(1,t) w_x(1,t) \\
=&~ E(1,t)w(1,t) - E(1,t)w(1,t) - \int_0^1E_x(x,t)w(x,t)dx \\
=&~ -\int_0^1 e(x,t)w(x,t)dx.
\end{align*}
Using $w=\theta e$ from Lemma \ref{lem1}, we have $we=\theta e^2$. Hence,
\begin{align*}
\dot{V}(t) = -\int_0^1 \theta(x,t) e(x,t)^2 dx. 
\end{align*}
Finally, using $\theta \geq \mu$ from Lemma \ref{lem1}, we obtain \eqref{V_dot}. 
\end{proof}

Now, using the inequality in Lemma \ref{lem2}, we would like to analyze $V$. To this end, we will need the following Poincaré-type inequality. 
 
\begin{lemma}\label{lem3}
Let $k>0$ and $\lambda$ be as in Theorem \ref{thm1}. For every $\phi \in \mathscr{C}^1([0,1])$
with $\phi(0)=0$,
\begin{align}\label{eq_poincare}
\int_0^1 \phi(x)^2dx + \frac{\phi(1)^2}{k} \leq \frac{1}{\lambda}\int_0^1 \phi'(x)^2dx .
\end{align}
\end{lemma}
\begin{proof}
Let $\xi$ be the unique solution in $(0,\pi/2)$ to $\xi \tan(\xi) = k$, and $\rho \in \mathscr{C}^1((0,1];\mathbb{R}_{>0})$ be given by
$$\rho(x):=\xi \cot(\xi x).$$
Using $\frac{d}{dx}\cot(\xi x) = -\xi\left(1+\cot^2(\xi x)\right)$, we have
\begin{align}\label{eq_riccati}
\rho'(x)+\rho(x)^2 = -\lambda \quad \forall x\in (0,1),
\end{align}
and
\begin{align}\label{eq_rho1}
\rho(1) = \xi \cot(\xi) = \frac{\xi^2}{\xi \tan \xi} = \frac{\lambda}{k} .
\end{align}
Next, let $\delta \in (0,1)$ and note that
\begin{align*}
0 \leq&~ \int_\delta^1 \left(\phi'(x) - \rho (x)\phi(x)\right)^2dx \\
=&~ \int_\delta^1 \phi'(x)^2 dx - \int_\delta^1 \rho(x)\left(\phi(x)^2\right)'dx + \int_\delta^1 \rho(x)^2\phi(x)^2dx \\
=&~ \int_\delta^1 \phi'(x)^2dx - \rho(1)\phi(1)^2 + \rho(\delta)\phi(\delta)^2 + \int_\delta^1 \left(\rho'(x)+\rho(x)^2\right)\phi(x)^2 dx .
\end{align*}
Substituting \eqref{eq_riccati} and \eqref{eq_rho1} and rearranging,
\begin{align}\label{eq_delta}
\int_\delta^1 \phi'(x)^2dx \geq \frac{\lambda}{k}\phi(1)^2 + \lambda\int_\delta^1\phi(x)^2dx
- \rho(\delta)\phi(\delta)^2 .
\end{align}
Since $\phi\in \mathscr{C}^1([0,1])$ and $\phi(0)=0$, we have
$$|\phi(\delta)| \leq \delta \max_{x\in [0,1]}|\phi'(x)|.$$
Moreover, $\cot(s) \leq 1/s$ for all
$s\in(0,\pi/2)$, so that 
$$\rho(\delta)\leq \frac{1}{\delta}.$$
Hence
$$\rho(\delta)\phi(\delta)^2 \leq \delta \left(\max_{x\in [0,1]}|\phi'(x)|\right)^2 \to 0 \ \ \text{as $\delta \to 0^+$}.$$
Letting $\delta \to 0^+$ in \eqref{eq_delta} and dividing by $\lambda$
yields \eqref{eq_poincare}.
\end{proof}

Using Lemma \ref{lem3}, we can prove the following result. 
\begin{lemma}\label{lem4}
For all $t\geq 0$, it holds that 
\begin{align}
2V(t) \leq \frac{1}{\lambda}\int_0^1e(x,t)^2dx. \label{V_bound}
\end{align}
\end{lemma}
\begin{proof}
Since $E(\cdot,t)\in \mathscr{C}^1([0,1])$ and $E(0,t)=0$ for each $t\geq 0$, we can apply Lemma \ref{lem3} with $\phi$ therein replaced by $E(\cdot,t)$, to obtain 
\begin{align*}
\int_0^1E(x,t)^2dx + \frac{E(1,t)^2}{k} \leq \frac{1}{\lambda}\int_0^1e(x,t)^2dx.
\end{align*}
\end{proof}

Combining Lemmas \ref{lem3} and \ref{lem4}, we can establish the following result, which will be key to prove Theorem \ref{thm1}. 
\begin{lemma}\label{lem5}
For all $t\geq 0$, we have  
\begin{align}
V(t) &\leq \exp\left(-2\lambda \mu t\right)V(0), \label{first_ineq} \\
\int_t^{+\infty}\int_0^1e(x,s)^2dx ds &\leq \frac{1}{\mu}\exp(-2\lambda \mu t)V(0). \label{second_ineq}
\end{align}
\end{lemma}
\begin{proof}
Using \eqref{V_dot} and \eqref{V_bound}, we have 
\begin{align*}
\dot{V}(t) \leq -2\lambda \mu V(t),
\end{align*}
which implies \eqref{first_ineq} by Grönwall's inequality. Furthermore, from \eqref{V_dot}, note that we have 
\begin{align}
\int_0^1e(x,t)^2dx \leq -\frac{1}{\mu}\dot{V}(t). \label{to_int}
\end{align}
Integrating both sides of \eqref{to_int} with respect to time and using \eqref{first_ineq} and the fact that $V(t) \to 0$ as $t\to +\infty$, we find
\begin{align*}
\int_t^{+\infty}\int_0^1e(x,s)^2dx ds &\leq \frac{1}{\mu} V(t) \nonumber \\
&\leq \frac{1}{\mu}\exp(-2\lambda \mu t) V(0).
\end{align*}
\end{proof}

A second key ingredient to prove Theorem \ref{thm1} is Lemma \ref{lem7} below, in which we establish $L^1$ global stability of $\{e=0\}$. In order to prove that lemma, we will need the following technical result. 
\begin{lemma}\label{lem6}
Let $\varepsilon >0$ and $\varphi_{\varepsilon} \in \mathscr{C}^{\infty}(\mathbb{R};\mathbb{R}_{\geq 0})$ be given by 
\begin{align}
\varphi_{\varepsilon}(s) := \sqrt{s^2+\varepsilon^2} - \varepsilon \quad \forall s\in \mathbb{R}. 
\end{align}
Then, the following properties hold. 
\begin{itemize}
    \item[(i)] For all $s\in \mathbb{R}$, we have
    \begin{align*}
    0 \leq |s| - \varphi_{\varepsilon}(s) \leq \varepsilon
    \end{align*}
    \item[(ii)] For all $s \in \mathbb{R}$, we have $|\varphi_{\varepsilon}'(s)| \leq 1$ and 
    \begin{align*}
        \lim_{\varepsilon \to 0}\varphi_{\varepsilon}'(s) = \sign(s). 
    \end{align*}
    \item[(iii)] For all $s\in \mathbb{R}$, we have $\varphi''_{\varepsilon}(s) >0$.
\end{itemize}
\end{lemma}
\begin{proof}
To prove (i), first note that 
\begin{align*}
s^2 \leq s^2 + \varepsilon^2 \leq \left(|s|+\varepsilon\right)^2.
\end{align*}
Hence, 
\begin{align*}
|s| \leq \sqrt{s^2+\varepsilon^2} \leq |s| + \varepsilon. 
\end{align*}
As a consequence, on the one hand, we have 
\begin{align*}
0\leq |s| - \left(\sqrt{s^2+\varepsilon^2} - \varepsilon\right) = |s| - \varphi_{\varepsilon}(s).
\end{align*}
On the other hand, we have 
\begin{align*}
|s|- \sqrt{s^2+\varepsilon^2} \leq 0,
\end{align*}
which implies that 
\begin{align*}
|s|-\varphi_{\varepsilon}(s) = |s| - \left(\sqrt{s^2+\varepsilon^2} - \varepsilon\right) \leq \varepsilon.
\end{align*}

Next, to prove (ii), note that 
\begin{align*}
\varphi_{\varepsilon}'(s) = \frac{s}{\sqrt{s^2+\varepsilon^2}}. 
\end{align*}
It implies that $|\varphi_{\varepsilon}'(s)| \leq 1$, and, for $s\neq 0$,
\begin{align*}
\lim_{\varepsilon \to 0} \varphi'_{\varepsilon}(s) = \frac{s}{|s|} = \sign(s). 
\end{align*}
When $s=0$, we have $\varphi_{\varepsilon}'(s) = \varphi'_{\varepsilon}(0) = 0 = \sign(0) = \sign(s)$. 

Finally, to prove (iii), it is sufficient to note that 
\begin{align*}
\varphi_{\varepsilon}''(s) = \frac{\varepsilon^2}{(s^2+\varepsilon^2)^{3/2}} >0.
\end{align*}
\end{proof}

We can now prove the following $L^1$ global stability result. 

\begin{lemma}\label{lem7}
For all $t_1,t_2\in \mathbb{R}_{\geq 0}$ with $t_1\leq t_2$, we have 
\begin{align}
\int_0^1|e(x,t_2)|dx \leq \int_0^1|e(x,t_1)|dx - k\int_{t_1}^{t_2}|w(1,t)|dt. \label{stab_L1}
\end{align}
\end{lemma}
\begin{proof}
Let $\varepsilon>0$ and $\varphi_{\varepsilon}$ be defined in Lemma \ref{lem6}. Moreover, let
\begin{align*}
L_{\varepsilon}(t) := \int_0^1 \varphi_{\varepsilon}\left(e(x,t)\right)dx \quad \forall t\geq 0. 
\end{align*}
Using (i) in Lemma \ref{lem6}, note that we have, for all $t\geq 0$, 
\begin{align*}
0 \leq L_{\varepsilon}(t) - \int_0^1 |e(x,t)|dx \leq \varepsilon.
\end{align*}
Hence, $L_\varepsilon(t) \to \int_0^1|e(x,t)|dx$ as $\varepsilon \to 0$, uniformly in $t$. Consequently, to establish \eqref{stab_L1}, we first analyze $L_\varepsilon$ and then we pass to the limit as $\varepsilon\to0$. 

Differentiating $L_{\varepsilon}$ with respect to time and using $\Sigma_e$, we obtain 
\begin{align*}
\dot{L}_{\varepsilon}(t) =&~ \int_0^1 \varphi_{\varepsilon}'(e(x,t))e_t(x,t)dx \\
=&~ \int_0^1 \varphi_{\varepsilon}'(e(x,t))w_{xx}(x,t)dx. 
\end{align*}
We rewrite the latter identity as 
\begin{align*}
\dot{L}_{\varepsilon}(t)  =&~ \int_0^1\varphi_{\varepsilon}'(w(x,t))w_{xx}(x,t)dx + \int_0^1\left(\varphi_{\varepsilon}'(e(x,t))-\varphi_{\varepsilon}'(w(x,t))\right)w_{xx}(x,t)dx. 
\end{align*}
Applying integration by parts, and using $w_x(1,t)=-kw(1,t)$, $w_x(0,t)=0$, and (iii), we obtain 
\begin{align*}
\int_0^1 \varphi_{\varepsilon}'(w(x,t))w_{xx}(x,t)dx =&~ \left[\varphi_{\varepsilon}'(w(x,t))w_{x}(x,t)\right]_{x=0}^{x=1} - \int_0^1 \varphi_{\varepsilon}''(w(x,t))w_{x}(x,t)^2 dx \\
\leq&~ - k\varphi_{\varepsilon}'(w(1,t))w(1,t).
\end{align*}
As a result, we have 
\begin{align*}
\dot{L}_{\varepsilon}(t) \leq&~ - k\varphi_{\varepsilon}'(w(1,t))w(1,t) +\int_0^1\left(\varphi_{\varepsilon}'(e(x,t))-\varphi_{\varepsilon}'(w(x,t))\right)w_{xx}(x,t)dx.
\end{align*}
Integrating in time from $t_1$ to $t_2$, we find 
\begin{align*}
L_{\varepsilon}(t_2) \leq&~ L_{\varepsilon}(t_1) - k\int_{t_1}^{t_2}\varphi_{\varepsilon}'(w(1,t))w(1,t)dt + \int_{t_1}^{t_2}\int_0^1\left(\varphi_{\varepsilon}'(e(x,t))-\varphi_{\varepsilon}'(w(x,t))\right)w_{xx}(x,t)dxdt. 
\end{align*}
We will now pass to the limit as $\varepsilon \to 0$. We have already shown that $L_{\varepsilon}(t) \to \int_{0}^1|e(x,t)|dx$ as $\varepsilon \to 0$, for all $t\geq 0$. Furthermore, from (ii) in Lemma \ref{lem6}, the fact that $t\mapsto w(1,t)$ is continuous hence bounded on the compact interval $[t_1,t_2]$, and the Lebesgue dominated convergence theorem, we have 
\begin{align*}
\int_{t_1}^{t_2}\varphi_{\varepsilon}'(w(1,t))w(1,t)dt  \to \int_{t_1}^{t_2}|w(1,t)|dt \ \ \text{as} \ \ \varepsilon \to 0. 
\end{align*}
Similarly, 
\begin{align*}
&\int_{t_1}^{t_2}\int_0^1\left(\varphi_{\varepsilon}'(e(x,t))-\varphi_{\varepsilon}'(w(x,t))\right)w_{xx}(x,t)dxdt \to \int_{t_1}^{t_2}\int_0^1\left(\sign(e(x,t))-\sign(w(x,t))\right)w_{xx}(x,t)dxdt
\end{align*}
as $\varepsilon \to 0$. However, from Lemma \ref{lem1}, we have $\sign(e)=\sign(w)$. Hence, 
\begin{align*}
\int_{t_1}^{t_2}\int_0^1\left(\varphi_{\varepsilon}'(e(x,t))-\varphi_{\varepsilon}'(w(x,t))\right)w_{xx}(x,t)dxdt \to 0
\end{align*}
as $\varepsilon \to 0$, which concludes the proof. 
\end{proof}

\subsection{Proof of Theorem \ref{thm1}}

Given Lemmas \ref{lem5} and \ref{lem7}, we can now prove Theorem \ref{thm1}. Define
\begin{align}
T^{\star} := \frac{1}{2\lambda \mu}, \label{Tstar}
\end{align} 
and let $t \geq T^{\star}$ and $s\in [t-T^{\star},t]$. Using Lemma \ref{lem7} and the
Cauchy-Schwarz inequality, we have
\begin{align*}
\int_0^1|e(x,t)|dx \leq \int_0^1|e(x,s)|dx
\leq \left(\int_0^1 e(x,s)^2dx\right)^{1/2}.
\end{align*}
The left-hand side of the latter inequality does not depend on $s$. Hence,
\begin{align*}
\left(\int_0^1|e(x,t)|dx\right)^2
\leq&~ \inf_{s\in [t-T^{\star},t]}\int_0^1e(x,s)^2dx \\
\leq&~ \frac{1}{T^{\star}}\int_{t-T^{\star}}^{t}\int_0^1e(x,s)^2dxds \\
\leq&~ \frac{1}{T^{\star}}\int_{t-T^{\star}}^{+\infty}\int_0^1e(x,s)^2dxds.
\end{align*}
Using \eqref{second_ineq} in Lemma \ref{lem5} and \eqref{Tstar}, it follows that
\begin{align*}
\left(\int_0^1|e(x,t)|dx\right)^2
\leq&~ \frac{1}{\mu T^{\star}}\exp\left(-2\lambda \mu \left(t-T^{\star}\right)\right)V(0) \\
=&~ 2\lambda \exp(1)\exp\left(-2\lambda \mu t\right)V(0).
\end{align*}
As a result, 
\begin{align*}
\int_0^1|e(x,t)|dx \leq \sqrt{2\lambda \exp(1)V(0)}\exp\left(-\lambda \mu t\right) \ \ \forall t\geq T^*. 
\end{align*}
However, note that, for all $x \in [0,1]$ and $t\geq 0$, we have
\begin{align*}
|E(x,t)| = \left|\int_0^xe(y,t)dy\right| \leq \int_0^1|e(y,t)|dy,
\end{align*}
and, in particular, $|M(t)| = |E(1,t)| \leq \int_0^1|e(y,t)|dy$. Substituting these two bounds
into the definition of $V$, we obtain
\begin{align*}
2V(t) \leq \left(\int_0^1|e(x,t)|dx\right)^2 + \frac{1}{k}\left(\int_0^1|e(x,t)|dx\right)^2.
\end{align*}
Hence, 
\begin{align*}
2V(t) \leq \frac{k+1}{k}\left(\int_0^1|e(x,t)|dx\right)^2 \ \ \forall t\geq 0.
\end{align*}
As a consequence, 
\begin{align*}
\int_0^1|e(x,t)|dx \leq&~ \sqrt{\frac{k+1}{k}\lambda \exp(1)} \int_0^1|e_o(x)|dx \times \exp\left(-\lambda \mu t\right) \ \ \forall t\geq T^*. 
\end{align*}
Using the relations $\lambda/k = \xi^2/(\xi \tan(\xi)) = \xi \cot(\xi)$ and $k=\xi \tan(\xi)$, we obtain  
\begin{align*}
\frac{k+1}{k} \lambda = \xi^2 + \xi \cot(\xi).
\end{align*}
As a result, using also Lemma \ref{lem7}, we conclude that \eqref{stab_ineq} holds for all $t\geq T^*$.

Let now $t \in [0,T^{\star})$. By \eqref{Tstar}, we have
\begin{align*}
C_1 \exp(-\lambda \mu t) > C_1\exp\left(-\lambda \mu T^{\star}\right)
= C_1 \exp\left(-\frac{1}{2}\right) > 1,
\end{align*}
so that $\min\left\{1,C_1\exp(-\lambda \mu t)\right\} = 1$, and \eqref{stab_ineq} reduces to
$\int_0^1|e(x,t)|dx \leq \int_0^1|e_o(x)|dx$, which is in Lemma \ref{lem7}.

\section{Example}

In this section, we evaluate Theorem \ref{thm1} via numerical simulations. 

\subsection{Numerical Scheme}

To simulate the fact that $\Sigma$ represents the physical continuous system, whereas $\hat{\Sigma}$ is the system implemented numerically in the calculator, we use a fine mesh to discretize $\Sigma$ and a coarser mesh to discretize $\hat{\Sigma}$. Accordingly, we consider two discretizations of the interval $[0,1]$: a set of points $\{x_i\}_{i=0}^{N}$ with $N\in \mathbb{N}^*$, $x_i=ih$ and $h=1/N$; and another set of points $\{\hat{x}_i\}_{i=0}^{\hat{N}}$ with $\hat{N} \in \mathbb{N}^*$, $\hat{x}_i = i \hat{h}$ and $\hat{h} = 1/\hat{N}$. We also consider two discretizations of the time interval $[0,t_f]$, where $t_f>0$ is the final time of simulation: a set of time instants $\{t_n\}_{n=0}^{N_t}$ with $N_t\in \mathbb{N}^*$, $t_n = n \Delta t$ and $\Delta t = t_f/N_t$; and another set of time instants $\{\hat{t}_m\}_{m=0}^{\hat{N}_t}$ with $\hat{N}_t = \frac{1}{r}N_t$ and $r\in \mathbb{N}^*$, $\hat{t}_m = m\widehat{\Delta t}$ and $\widehat{\Delta t} = t_f/\hat{N}_t = r\Delta t$. We discretize $\Sigma$ on $\{x_i\}_{i=0}^N\times \{t_n\}_{n=0}^{N_t}$, and we discretize $\hat{\Sigma}$ on $\{\hat{x}_i\}_{i=0}^{\hat{N}}\times \{\hat{t}_m\}_{m=0}^{\hat{N}_t}$. 

We start by describing the numerical scheme used to simulate $\Sigma$. Letting $\alpha(z) := \Psi'(z)$, we need to approximate 
$$\left(\alpha(z)z_x\right)_x = \alpha(z)_x z_x + \alpha(z)z_{xx}.$$
To this end, we employ a central difference scheme to approximate $\alpha (z(x_i))z_{xx}(x_i)$, i.e., 
\begin{align}
\alpha(z(x_i))z_{xx}(x_i) \approx \alpha(z_i)&\frac{z_{i+1}-2z_i+z_{i-1}}{h^2} \quad \forall i\in \{1,2,...,N-1\}, \label{approx_1}
\end{align}
where $z_i \approx z(x_i)$. Furthermore, we employ an Euler backward scheme to approximate $\alpha(z)_xz_x\big|_{x=x_i}$, i.e., 
\begin{align}
\alpha(z)_xz_x\big|_{x=x_i} \approx &\frac{(\alpha(z_i)-\alpha(z_{i-1}))(z_i-z_{i-1})}{h^2} \quad \forall i\in \{1,2,...,N-1\}.  \label{approx_2}
\end{align}
Combining \eqref{approx_1} and \eqref{approx_2}, we obtain 
\begin{align}
\left(\alpha(z)z_x\right)_x \big|_{x=x_i} \approx \alpha(z_i)&\frac{z_{i+1}-z_i}{h^2}-\alpha(z_{i-1})\frac{z_i-z_{i-1}}{h^2} \quad \forall i\in \{1,2,...,N-1\}. \nonumber 
\end{align}
To approximate the boundary conditions we let 
\begin{align*}
z_0 = z_1 \quad \text{and} \quad z_N=z_{N-1}. 
\end{align*}

For time discretization, we use a semi-implicit Euler backward scheme. That is, we let
\begin{align*}
\frac{z_i^{n+1}-z_i^n}{\Delta t} =&~ \alpha(z_i^n)\frac{z^{n+1}_{i+1}-z^{n+1}_i}{h^2}-\alpha(z^n_{i-1})\frac{z^{n+1}_i-z^{n+1}_{i-1}}{h^2} \quad \forall i\in \{1,2,...,N-1\},
\end{align*}
where $z_i^n \approx z_i(t_n)$.

As a result, $\Sigma$ is approximated by the tridiagonal system
\begin{equation*}
\Sigma_d : \left\lbrace 
\begin{aligned}
&\begin{bmatrix}
A_1(z_i^n) & A_2(z_i^n,z_{i-1}^n) & A_1(z_{i-1}^n)
\end{bmatrix} \begin{bmatrix}
z_{i+1}^{n+1} \\ z_i^{n+1} \\ z_{i-1}^{n+1}
\end{bmatrix} = z_i^n \quad \forall i\in \{1,2,...,N-1\}, \\
&z_0^{n+1} = z_1^{n+1}, \\
&z_N^{n+1} = z_{N-1}^{n+1},
\end{aligned}
\right.
\end{equation*}
where 
\begin{align*}
A_1(z_i^n) &:= -\frac{\Delta t}{h^2}\alpha (z_i^n),  \\
A_2(z_i^n,z_{i-1}^n) &:= 1 + \frac{\Delta t}{h^2}\left(\alpha (z_i^n) + \alpha (z_{i-1}^n)\right).
\end{align*}
This system is solved using the Thomas algorithm.

We turn now to the approximation of $\hat{\Sigma}$. Since $\widehat{\Delta t}$ is a multiple of $\Delta t$, we can write
\begin{align*}
\frac{\hat{z}_j^{m+1}-\hat{z}_j^m}{\widehat{\Delta t}} =&~ \alpha(\hat{z}_j^m)\frac{\hat{z}^{m+1}_{j+1}-\hat{z}^{m+1}_j}{\hat{h}^2}-\alpha(\hat{z}^m_{j-1})\frac{\hat{z}^{m+1}_j-\hat{z}^{m+1}_{j-1}}{\hat{h}^2} \quad \forall j\in \{1,2,...,\hat{N}-1\},
\end{align*}
where $\hat{z}_j^m \approx \hat{z}(\hat{x}_j,\hat{t}_m)$.

The left-end boundary condition is approximated by setting 
$$ \hat{z}_0^{m+1} = \hat{z}_1^{m+1}.$$
The right-end boundary condition is approximated by 
\begin{align*}
\alpha(\hat{z}_{\hat{N}}^m)\frac{\hat{z}_{\hat{N}}^{m+1}-\hat{z}_{\hat{N}-1}^{m+1}}{\hat{h}} = k\left(\Psi(z_N^{rm})-\Psi(\hat{z}_{\hat{N}}^m)\right). 
\end{align*}
As a result, $\hat{\Sigma}$ is approximated by the tridiagonal system 
\begin{equation*}
\hat{\Sigma}_d : \left\lbrace 
\begin{aligned}
&\begin{bmatrix}
A_1(\hat{z}_j^m) & A_2(\hat{z}_j^m,\hat{z}_{j-1}^m) & A_1(\hat{z}_{j-1}^m)
\end{bmatrix} \begin{bmatrix}
\hat{z}_{j+1}^{m+1} \\ \hat{z}_j^{m+1} \\ \hat{z}_{j-1}^{m+1}
\end{bmatrix} = \hat{z}_j^m \quad \forall j\in \{1,2,...,\hat{N}-1\}, \\
&\hat{z}_0^{m+1} = \hat{z}_1^{m+1}, \\
&\hat{z}_{\hat{N}}^{m+1} = \hat{z}_{\hat{N}-1}^{m+1} + k\hat{h} \left(\frac{\Psi(z_{N}^{rm})-\Psi(\hat{z}_{\hat{N}}^{m})}{\alpha(\hat{z}_{\hat{N}}^m)}\right),
\end{aligned}
\right.
\end{equation*}
that is solved using the Thomas algorithm. 

We set $N=60$, $\Delta t = 10^{-4}$, $\hat{N}=30$, and 
$\widehat{\Delta t} = 2\Delta t$.

\subsection{Simulation Results}

Let 
\begin{align*}
\Psi(z) := \int_0^z (s^2+1)^{5/4}ds, 
\end{align*}
so that 
\begin{align*}
\Psi'(z) = (z^2+1)^{5/4} \geq 1 >0. 
\end{align*}
The resulting equation $\Sigma$ models ion temperature along a magnetic field line in the scrape-off layer of a tokamak \cite[Pages 652--684]{plasma}.

The system's initial condition is taken as
\begin{align}\label{eq_zo_sim}
z_o(x) := 1 + 0.3\cos(\pi x),
\end{align}
which is smooth and satisfies $z_o'(0)=z_o'(1)=0$, hence
$\Psi(z_o)_x|_{x=0} = \Psi(z_o)_x|_{x=1} = 0$.

The observer's initial condition is set as 
\begin{align}\label{eq_zho_sim}
\hat{z}_o := \Psi^{-1}\Big( \Psi(z_o) - A\cos(\xi x) \Big),
\end{align}
where $A\in\mathbb{R}$, and $\xi$ is the unique solution in $ (0,\pi/2)$ to $\xi\tan(\xi)=k$. One can readily verify that $\hat{z}_o$ satisfies the required regularity and compatibility conditions. 

Figures \ref{fig1} shows the function $t\mapsto |e(\cdot,t)|_{L^1}$, for $k=1$ and $A \in \{20,200,2000\}$, on a semi-logarithmic scale. As expected from Theorem \ref{thm1}, the norm of the estimation error converge exponentially to zero.  
\begin{figure}
    \centering
    \includegraphics[width=\linewidth]{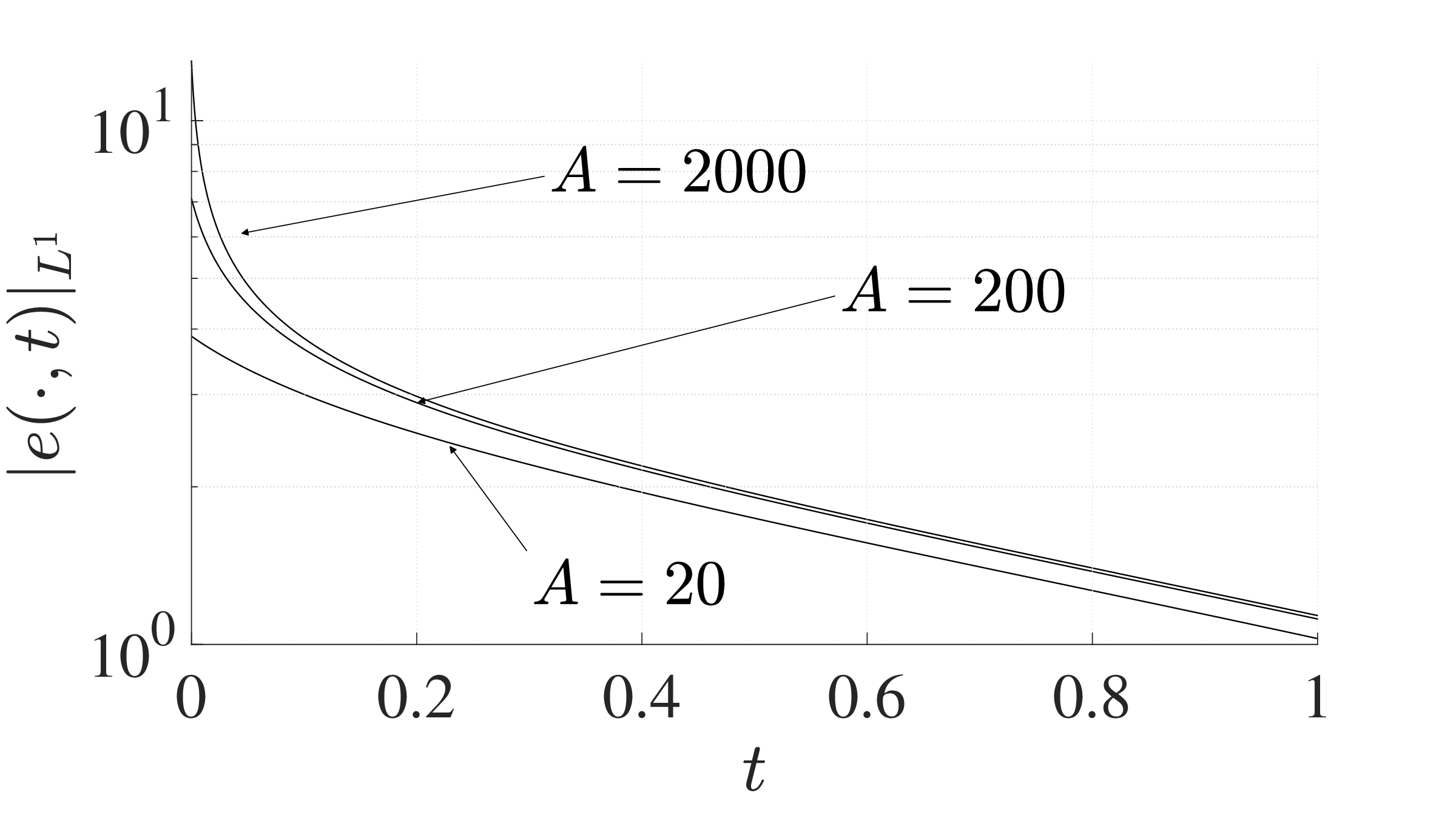}
    \caption{The $L^1$ norm of the estimation error for different sizes of the initial estimation error.}
    \label{fig1}
\end{figure}

Next, Figure \ref{fig3} shows $|e(\cdot,t)|_{L^1}/|e_o|_{L^1}$ for $k \in \{1,10,20,30\}$, with $A = 20$. It can be noted that the convergence rate increases with $k$, but the improvement is diminishing: the gap between $k=1$ and $k=10$ is larger than the gap between $k=20$ and $k=30$. This is expected from the fact that $\lambda \in (0,\pi^2/4)$ in Remark \ref{rem1}.

\begin{figure}
    \centering
    \includegraphics[width=\linewidth]{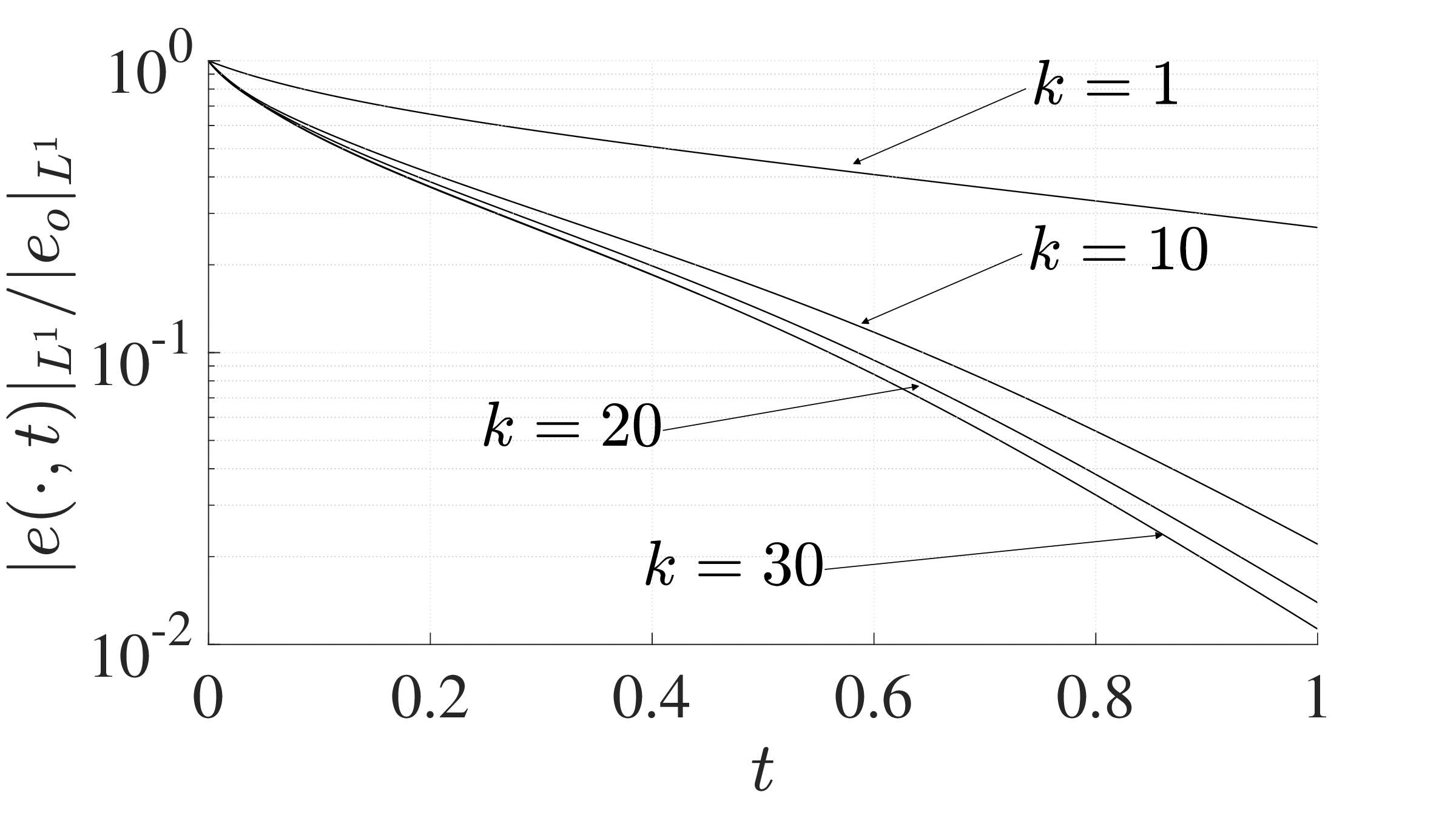}
    \caption{The ratio of the $L^1$ norm of the estimation error over the $L^1$ norm of the initial estimation error, for different values of the observer gain $k$.}
    \label{fig3}
\end{figure}

Finally, Figure \ref{fig4} shows $t\mapsto |e(\cdot,t)|_{L^1}$ for $k=5$, $A=2$ and an output corrupted by
additive noise,
\begin{align*}
y_{\rm meas}(t) = y(t)+\nu(t),
\end{align*}
with $\nu$ drawn independently at each instant from a centred Gaussian distribution of
standard deviation $\sigma_{\rm n} \in \{0,10^{-2},10^{-1}\}$. Theorem \ref{thm1} does not cover this case, but by adapting the proof of Lemma \ref{lem7}, one can readily show an $L^1$-input-to-state stability result with respect to the noise, which is in agreement with Figure \ref{fig4}.

\begin{figure}
    \centering
    \includegraphics[width=\linewidth]{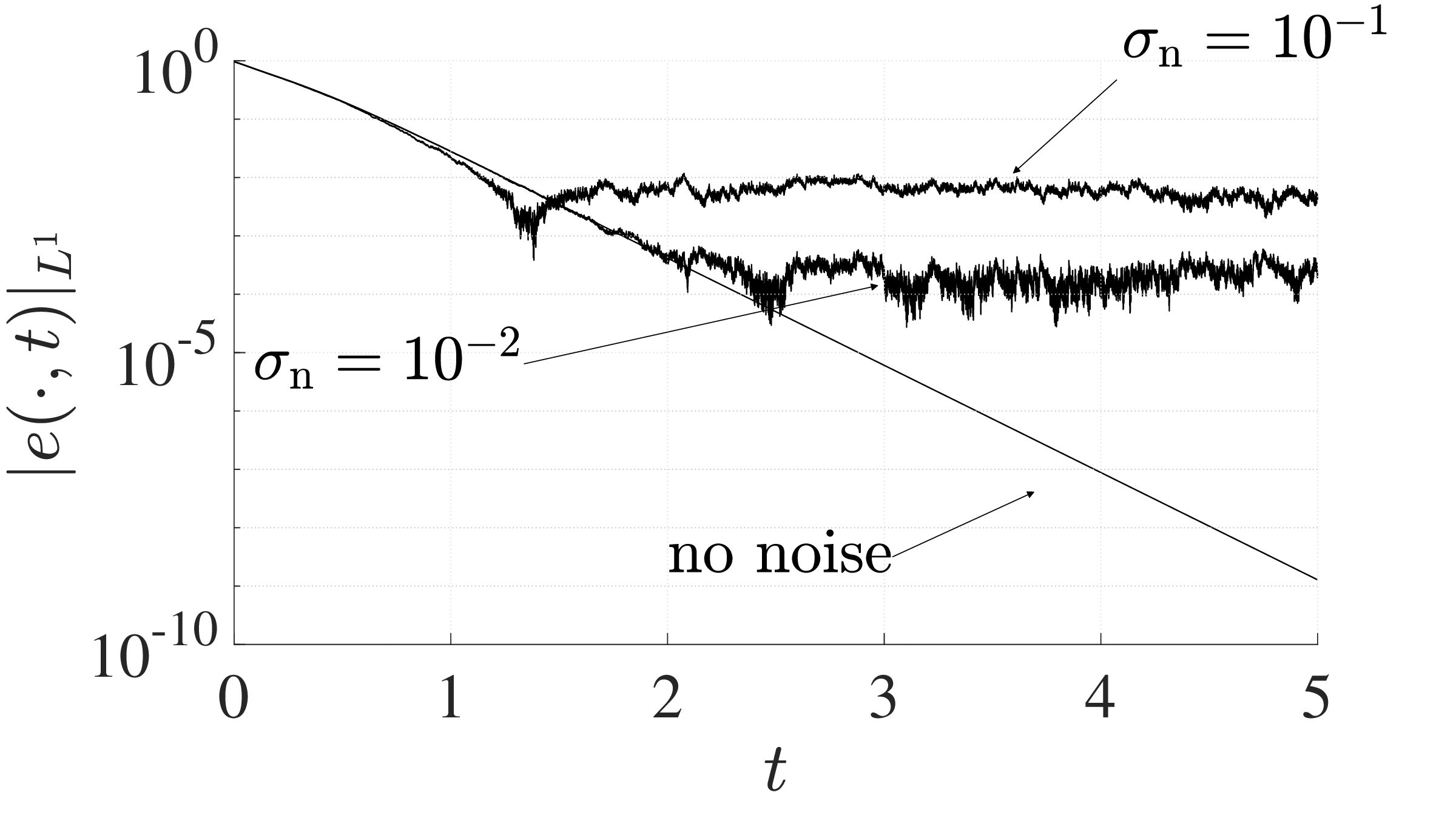}
    \caption{The $L^1$ norm of the estimation error in the presence of output noise.}
    \label{fig4}
\end{figure}
\section{Conclusion and Perspectives}

We have designed a nonlinear observer for a one-dimensional quasilinear diffusion equation with Neumann boundary conditions, using the state at one end of the spatial domain as an output. We have proved global exponential stability of the zero equilibrium for the estimation error dynamics in the $L^1$ norm. The result holds for a general nonlinear diffusivity, under the only assumption that it is bounded from below by a strictly positive constant.

Several directions remain to be considered. First, it would be interesting to establish global exponential stability in the $L^p$ norm for $p>1$. Moreover, extension to higher-dimensional domains is important. Finally, extending the results from classical solutions to weak solutions would require alternative arguments at some steps of the proof.

\section{Acknowledgement}
The support of a grant from National Science and Engineering Research Council of Canada for this research is gratefully acknowledged.

\end{document}